\documentclass[a4paper,english,leqno]{amsart}

\title[]{Solution theory to\\semilinear stochastic equations of Schr\"odinger type on curved spaces\\I - Operators with uniformly bounded coefficients}

\author{Alessia Ascanelli}
\address{Dipartimento di Matematica ed Informatica, Universit\`a di Ferrara, Via Machiavelli n.~30, 44121 Ferrara, Italy}
\email{alessia.ascanelli@unife.it}

\author{Sandro Coriasco}
\address{Dipartimento di Matematica ``G. Peano'', Universit\`a degli Studi di Torino, via Carlo Alberto n.~10, 10123 Torino, Italy}
\email{sandro.coriasco@unito.it}

\author{Andr{\'e} S{\"u}\ss}
\address{C/O Dipartimento di Matematica ed Informatica, Universit\`a di Ferrara, Via Machiavelli n.~30, 44121 Ferrara, Italy}
\email{suess.andre@web.de}

\date{}

\usepackage[utf8]{inputenc}
\usepackage{mathrsfs}
\usepackage{babel}
\usepackage{amsmath,amsthm,amssymb,xy,amsxtra,latexsym,verbatim}
\usepackage{hyperref}
\usepackage{enumitem}
\usepackage{pdfsync}
\usepackage{amsfonts}
\usepackage{colortbl}
\usepackage{url}

\usepackage{amsxtra}
\usepackage{pxfonts}

\newcommand*{\ii}{\mathrm{i}}

\newcommand*{\scrF}{\ensuremath{\mathscr{F}}}	

\newcommand*{\caF}{\ensuremath{\mathcal{F}}}		
\newcommand*{\caH}{\ensuremath{\mathcal{H}}}
\newcommand*{\caS}{\ensuremath{\mathcal{S}}}		

\def\ds{\displaystyle}

\newcommand*{\N}{\mathbb{N}}										
\newcommand*{\Z}{\mathbb{Z}}										
\newcommand*{\R}{\mathbb{R}}										
\newcommand*{\Rd}{{\mathbb{R}^d}}							
\newcommand*{\C}{\mathbb{C}}

\newcommand*{\E}{\mathbb{E}}										
\renewcommand*{\P}{\mathbb{P}}									

\def\<{{\langle}}
\def\>{{\rangle}}

\def\BBB{\mathfrak{B}}

\newcommand{\Lip}{\mathcal{L}ip}
\newcommand{\Liploc}{\mathcal{L}ip_{\mathrm{loc}}}
\newcommand{\Hzzeta}{\mathcal{H}_{z,\zeta}}
\newcommand{\HzzetaRd}{\mathcal{H}_{z,\zeta}(\R^d)}

\numberwithin{equation}{section}
\allowdisplaybreaks

\theoremstyle{plain}
\newtheorem{lemma}{Lemma}[section]
\newtheorem{theorem}[lemma]{Theorem}

\theoremstyle{definition}
\newtheorem{definition}[lemma]{Definition}
\newtheorem{remark}[lemma]{Remark}
\newtheorem{example}[lemma]{Example}

\newcommand{\beqsn}{\arraycolsep1.5pt\begin{eqnarray*}}
\newcommand{\eeqsn}{\end{eqnarray*}\arraycolsep5pt}
\newcommand{\beqs}{\arraycolsep1.5pt\begin{eqnarray}}
\newcommand{\eeqs}{\end{eqnarray}\arraycolsep5pt}

\usepackage{xcolor}
\definecolor{red}{rgb}{1,0,0}

\def\Op{ {\operatorname{Op}} }

\begin{document}

\begin{abstract}
We study the Cauchy problem for Schr\"odinger type stochastic partial differential equations with uniformly bounded coefficients on a curved space. We give conditions on the coefficients, on the drift and diffusion terms, on the Cauchy data, and on the spectral measure associated with the noise, such that the Cauchy problem admits a unique function-valued mild solution in the sense of Da Prato and Zabczyc.
\end{abstract}



\maketitle

%
\section{Introduction}\label{sec:intro}

In this paper we are interested in the Cauchy problem for a semilinear stochastic equation of Schr\"odinger type, that is
\begin{equation}\label{eq:SPDE}
\left\{
\begin{array}{l} 
P(x,\partial_t,\partial_x) u(t,x)= \gamma(t,x,u(t,x)) + \sigma(t,x, u(t,x))\dot{\Xi}(t,x), \quad (t,x)\in[0,T]\times\R^d,
\\
u(0,x)=u_0(x),\quad x\in\R^d,
\end{array}\right.
\end{equation}
where: 
\begin{itemize} 
\item $P$ is a Schr\"odinger type operator on a curved space of the form considered, e.g., by Craig (see \cite{craig} and the literature mentioned therein), 
namely,
\begin{equation}\label{operator}
\begin{aligned}
 P(x,\partial_t,\partial_x) &=i\partial_t +\ds\frac12\sum_{j,\ell=1}^d\partial_{x_j}\left(a_{j\ell}(x)\partial_{x_\ell}\right)+m_1(x,-i\partial_x)+m_0(x,-i\partial_x)
 \\
 &=i\partial_t +a(x,D_x) +a_1(x,D_x) +m_1(x,D_x)+m_0(x,D_x),
 \end{aligned}
\end{equation}
where, having set, as usual, $D_x=-i\partial_x$:
\\
$a(x,\xi):=-\ds\frac12\sum_{j,\ell=1}^d a_{j\ell}(x)\xi_j\xi_\ell$, $a_{j\ell}=a_{\ell j}$, $j,\ell=1,\dots,d$, is the \textit{Hamiltonian} of the equation, \\
$a_1(x,\xi) := \ds\frac i2 \ds\sum_{j,\ell=1}^d \partial_{x_j}a_{j\ell}(x)\xi_\ell$, $m_1(x,\xi)$ comes from a magnetic field and $m_0(x,\xi)$ is a potential term;
\item $\gamma$ and $\sigma$, respectively the drift term and the diffusion coefficient, are real-valued functions, subject to certain regularity conditions
(see below);
\item $\Xi$ is an $\mathcal S'(\R^d)$-valued Gaussian process, white in time and coloured in space, with \textit{correlation measure} $\Gamma$ and \textit{spectral measure}
$\mathfrak{M}$ (see Section \ref{sec:noise} for the precise definition);
\item $u$ is an unknown stochastic process, called \emph{solution} of the Cauchy problem \eqref{eq:SPDE}.
\end{itemize}

\medskip

 To give meaning to \eqref{eq:SPDE} we rewrite it formally in its corresponding integral form and look for \emph{mild solutions}, that is, stochastic processes $u(t)$ satisfying an integral equation of the form
\begin{equation}\label{eq:mildsolutionSPDE}
  u(t) = S(t)u_0-i\int_0^tS(t-s)\gamma(s,u(s))ds -i\int_0^tS(t-s)\sigma(s,u(s))d\Xi(s), \quad \forall t\in [0,T_0],\ 0<T_0\leq T,\ x\in\R^d,
  \end{equation}
where $S(t)$ is the propagator of the evolution operator $P$, that is, a family of operators depending on the parameter $t\in[0,T_0]$ such that, for every $t\in[0,T_0]$, it holds $P(x,\partial_t,\partial_x)\circ S(t)=0$ and $S(0)=Id.$ Note that the first integral in \eqref{eq:mildsolutionSPDE} is of deterministic type, while the second is a stochastic integral.
The kind of solution $u$ we can construct for the Cauchy problem \eqref{eq:SPDE} depends on how one makes sense of the stochastic integral appearing in \eqref{eq:mildsolutionSPDE}. Here we follow the approach by Da Prato and Zabczyk  (see \cite{dapratozabczyk}), which consists in associating a $\caH$-valued Brownian motion with the random noise, 
where $\caH$ is an appropriately chosen Hilbert space. One can then define the stochastic integral as an infinite sum of It\^{o} integrals with respect to one-dimensional Brownian motions. This leads to solutions involving random functions taking values in $\caH$.

In literature, the existence of a unique solution to the Cauchy problem for an SPDE is often stated under suitable conditions on the coefficients and a \emph{compatibility condition} between the noise $\Xi$ and the equation in \eqref{eq:SPDE}, expressed in terms of integrals with respect to the spectral measure $\mathfrak M$ associated with the noise (see Section \ref{sec:noise}). Recently, we studied a class of stochastic PDEs with $(t,x)$-depending unbounded coefficients, admitting, at most, a polynomial growth as $|x|\to\infty$. We dealt both with hyperbolic and parabolic type operators, constructing a solution theory for the associated classes of Cauchy problems, see \cite{ACS19b, semilinearpara}. Now, we want to investigate the case of Schr\"odinger type equations on a curved space, whose theory (also in the deterministic case) is far from being straightforward (see, e.g., 
\cite{craig, wunsch, itonakamura, nicola, yajima}). In several applications of Schrödinger equation, a random potential appears: when this potential depends on $x$, it has an effect on the dynamics of the solution. For instance, focusing/defocusing Schrödinger equations with a white in time and coloured in space noise are studied in $L^p_x$- based Sobolev spaces in \cite{DbD99, DbD2002, DbD2005, oh}, together with the stopping time for the solution.

In this paper we start by
considering an operator $P$ of the form \eqref{operator} with uniformly bounded coefficients. Our analysis will then continue, considering the case of potentials with polynomial growth, 
in the forthcoming paper \cite{fantasma}. For such reason, in this series of papers we adopt a unified treatment, employing the class $S^{m,\mu}(\R^d)$, $m,\mu\in\R$, of symbols of 
order $(m,\mu)$, given by the set of all functions $a\in {C}^\infty(\R^d\times\R^d)$ satisfying, for every $\alpha,\beta\in \Z^n_+$, the global estimates
$$|\partial_\xi^\alpha\partial_x^\beta a(x,\xi)|\leq C_{\alpha,\beta}\langle x\rangle^{m-|\beta|}\langle\xi\rangle^{\mu-|\alpha|},\quad (x,\xi)\in \R^{2d},$$
for suitable constants $C_{\alpha\beta}>0$. Recall that, with any symbol $a\in S^{m,\mu}(\R^d)$, it is associated a pseudodifferential operator  
\begin{equation}\label{eq:psdo}
[\Op(a)u](x)=[a(\cdot,D)u](x)=(2\pi)^{-d}\int e^{\ii x\xi}a(x,\xi)\hat{u}(\xi)d\xi, \quad u\in\caS(\R^d),
\end{equation}
linear and continuous from $\caS(\R^d)$ to itself, extended by duality to a linear continuous operator from $\caS^\prime(\R^d)$ to itself 
(see, e.g., the introductory sections of \cite{ACS19b,ACS19a,linearpara,semilinearpara} and \cite{cordes}, for details about the associated calculus).
Moreover, $\Op(a)$ acts continuously from $H^{r,\rho}(\R^d)$ to $H^{r-m,\rho-\mu}(\R^d)$,
where, given $r,\rho\in\R$, the \textit{weighted Sobolev} spaces $H^{r,\rho}(\R^d)$ (also known as Sobolev-Kato spaces) are defined by
$$H^{r,\rho}(\R^d):=\{u\in\mathcal S'(\R^d)\vert\ \Op(\lambda_{r,\rho}) u\in L^2(\R^d)\}, \quad \lambda_{r,\rho}(x,\xi)=\langle x\rangle^{r}\langle\xi\rangle^{\rho}.$$

\medskip
Here we assume the following hypotheses on the operator $P$ (see \cite{craig}):
\begin{enumerate}
\item\label{ctr:aorderzero} the \textit{Hamiltonian} satisfies $a\in S^{0,2}(\R^d)$;
\item \label{ctr:aorder1}the \textit{lower order metric terms} satisfy $a_1\in S^{-1,1}(\R^d)$;
\item the \textit{magnetic field term} satisfies $m_1\in S^{0,1}(\R^d)$ and is real-valued;
\item the \textit{potential} satisfies $m_0\in S^{0,0}(\R^d)$;
\item \label{ctr:ell} $a$ satisfies, for all $x,\xi\in\R^d$, $C^{-1}|\xi|^2\leq a(x,\xi)\leq C|\xi|^2$. 
\end{enumerate}
\begin{remark}
	\begin{enumerate}
	\item[(i)] In the sequel, we will often omit the base spaces $\R^d$, $\R^{2d}$, etc., from the notation.
	\item [(ii)] The symbol spaces $S^{m,\mu}$ are denoted by $S^{\mu,m}(1,0)$ in \cite{craig}, where it is remarked that the 
	ellipticity condition \eqref{ctr:ell}, together with the other hypotheses on $a$ and $a_1$,
	implies that the matrix $(a_{j\ell})$ is invertible, as well as that the Riemannian metric given by 
	the matrix $(a_{j\ell})^{-1}=(a^{j\ell})=\mathfrak{a}$ is asymptotically flat.
	\item[(iii)] Notice that, in view of the hypotheses on $\mathfrak{a}$, our analyis actually covers the case
	\[
		P(x,\partial_t,\partial_x)=i\partial_t+\frac{1}{2}\Delta_a+\widetilde{m}_1(x,\partial_x)+m_0(x,\partial_x),
	\]
	where $\widetilde m_1\in S^{0,1}$, $m_0\in S^{0,0}$, and 
	\[
		\Delta_a=\sqrt{\det(\mathfrak{a})}\sum_{j,\ell=1}^d
		\partial_{x_j}\left[\sqrt{\det(\mathfrak{a})^{-1}}\,a^{j\ell}\,
		\partial_{x_\ell}\right]
	\]
	is the Laplace-Beltrami operator associated with $\mathfrak{a}$. Indeed, \eqref{ctr:aorderzero}, \eqref{ctr:aorder1} and \eqref{ctr:ell} imply that $a_{j\ell}\in S^{0,0}$ and $\det(\mathfrak{a})\ge c>0$. In turn, this implies
	$\det(\mathfrak{a}), \det(\mathfrak{a})^{-1}, a^{j\ell}, \sqrt{\det(\mathfrak{a})}, \sqrt{\det(\mathfrak{a})^{-1}}\in S^{0,0}$ and
	\begin{align*}
	 	\Delta_a&= \sqrt{\det(\mathfrak{a})}\sum_{j,\ell=1}^d
		\partial_{x_j}\left[\sqrt{\det(\mathfrak{a})^{-1}}\,a^{j\ell}\,
		\partial_{x_\ell}\right]
		= \sum_{j,\ell=1}^d
		a^{j\ell}\partial_{x_j}\partial_{x_\ell}
		+ \sum_{j,\ell=1}^d \sqrt{\det(\mathfrak{a})}
		\left\{\partial_{x_j} \left[
		\sqrt{\det(\mathfrak{a})^{-1}}\, a^{j\ell}
		\right]
		\right\}\partial_{x_\ell}
		\\
		&= \sum_{j,\ell=1}^d
		a^{j\ell}\partial_{x_j}\partial_{x_\ell}
		\mod \Op(S^{-1,1}). 
	\end{align*}
	Then, the \textit{non-selfadjoint terms} can be included into 
	$\widetilde{m}_1\in S^{0,1}\supset S^{-1,1}$, see \cite[p.XX-4]{craig}.
	\end{enumerate}
\end{remark}

To state the main result of the present paper we need to introduce a subclass of the Sobolev-Kato spaces and a class of Lipschitz functions
(the latter, analogous to those appearing in \cite{ACS19b}).
\begin{definition}\label{def:Hpint}
	Given $z\in\N$, $\zeta\in\R$, set $\HzzetaRd :=\ds\bigcap_{j=0}^z H^{z-j,j+\zeta}(\R^d)$.
	The space $\mathcal H_{z,\zeta}(\R^d)$ is endowed with the norm 
	\begin{equation}\label{eq:normHpint}
		\|u\|_{\HzzetaRd}:=\ds\sum_{j=0}^z \|u\|_{H^{z-j,j+\zeta}(\R^d)}.
	\end{equation}
\end{definition}
\noindent
By the immersion properties of the Sobolev-Kato spaces, we immediately see that 
$H^{z,z+\zeta}(\R^d)\subset \mathcal H_{z,\zeta}(\R^d)\subset H^{z,\zeta}(\R^d)$.

\begin{remark}\label{rem:Hpint}
\begin{itemize}
\item[(i)]Since every space $H^{r,\rho}$ with $r\geq 0$ and $\rho>d/2$ is an algebra (see \cite[Proposition 2.2]{AC06}), also 
$\Hzzeta$ is an algebra for $\zeta>d/2$ (see Example \ref{ex:power} below). 
\item[(ii)]The Hilbert spaces based on the norm \eqref{eq:normHpint} for an arbitrary $\zeta\in\N$ are mentioned in \cite[Page XX-12]{craig}, where, in particular,
	the \textit{unweighted} Sobolev spaces $H^{0,\rho}$ is denoted, as usual, by $H^\rho$, and $\caH_{r,0}$, 
	the space of spatial moments up to order $r\in\N$, is denoted by $W^r$.
	\end{itemize}
\end{remark}

In the linear deterministic case, that is, for $\sigma=\gamma\equiv0$, the existence of a unique solution to the Cauchy problem \eqref{eq:SPDE} and the evolution of its solution 
has been fully described.
\begin{theorem}[{\cite[Page XX-12]{craig}}]\label{T14craig}
Under the assumptions (1)-(5), the solution $u(t)$ to the Cauchy problem \eqref{eq:SPDE} with $\sigma=\gamma\equiv 0$ satisfies the estimate
\[\|u(t)\|_{\mathcal H_{z,\zeta}(\R^d)}\leq e^{C_{z,\zeta}t}\|u_0\|_{\mathcal H_{z,\zeta}(\R^d)},\quad t\in[0,T_0], \]
for $T_0\in(0,T]$ and a positive constant $C_{z,\zeta}$ depending only on $z,\zeta\in\N$.
\end{theorem}

\begin{remark}\label{S}
As a consequence of Theorem \ref{T14craig}, the \textit{propagator} $S$ (or, equivalently, the fundamental solution) of $P$ defines continuous maps $S(t):\Hzzeta\rightarrow \Hzzeta$, whose norms can be bounded by $e^{C_{z,\zeta}t}$, $t\in[0,T_0]$, $z,\zeta\in\N$.
\end{remark}

\begin{definition}\label{def:lip} Given $z\in\N$, $\zeta\in[0,+\infty)$, $\Lip(z,\zeta)$ is the set of all measurable functions 
$g:[0,T]\times\R^d\times\C\longrightarrow\C$
such that there exists a real-valued, non negative,
$C_t=C(t)\in C([0,T])$, such that:
\begin{itemize}
\item for every $v\in \HzzetaRd$ and $t\in[0,T]$, we have $\|g(t,\cdot,v)\|_{\HzzetaRd}\leq C(t)\left[1+\|v\|_{\HzzetaRd}\right]$;
%
%
\item for every $v_1,v_2\in \HzzetaRd$ and $t\in[0,T]$, we have $\|g(t,\cdot,v_1)-g(t,\cdot,v_2)\|_{\HzzetaRd}\leq C(t)\|v_1-v_2\|_{\HzzetaRd}$.
\end{itemize}
\end{definition}
\noindent
More generally, we say that $g\in\Liploc(z,\zeta)$ if the stated properties hold true for $v,v_1,v_2\in U$,  with $U$ a suitable open subset of $\HzzetaRd$ 
(typically, a sufficiently small neighbourhood of the initial data of the Cauchy problem).

\medskip

The main result of the present paper is the following one, see Section \ref{sec:stochastics} and Theorem \ref{main} for the details:

\medskip

\emph{%
\noindent
Let us consider the Cauchy problem \eqref{eq:SPDE} for a 
Schr\"odinger type operator \eqref{operator} under assumptions (1)-(5), and suppose $u_0\in \HzzetaRd$, $z,\zeta\in\N$. Moreover, assume that $\gamma,\sigma\in\Liploc(z,\zeta)$ in some open subset $U\subset \HzzetaRd$ with $u_0 \in U$, and 
         \beqs\label{condmeas}
         \int_\Rd {\mathfrak{M}(d\xi)}<\infty.
	\eeqs
Then, there exists a time horizon $0< T_0\leq T$ such that the Cauchy problem
\eqref{eq:SPDE} admits a unique solution $u\in L^2([0,T_0]\times\Omega, \mathcal H_{z,\zeta}(\R^d))$ satisfying \eqref{eq:mildsolutionSPDE}, where the first integral is a Bochner integral, and the second integral is understood as stochastic integral of a suitable $\mathcal H_{z,\zeta}(\R^d)$-valued stochastic process with respect to a cylindrical Wiener process associated with the stochastic noise $\Xi$.
}

\begin{example}\label{ex:power}
	$g(t, x, u)=u^n$, $n\in\N$, is an admissible non-linearity in the Cauchy problem \eqref{eq:SPDE} when $\zeta>d/2$. In fact, $g\in\Liploc(z,\zeta)$, 
	when $z\in\N$, $\zeta>d/2$, since, if $v\in \Hzzeta$ is such that $\|v\|_{\Hzzeta}\le R$, then
	\begin{equation}\label{eq:powerliploc}
		\|g(t,x,v)\|_{\Hzzeta}=\| v^n \|_{\Hzzeta}\le \widetilde{C} R^{n-1} \|v\|_{\Hzzeta}
		\le \left( \widetilde{C} R^{n-1} \right) \left[1+\|v\|_{\Hzzeta}\right].
	\end{equation}
	Indeed, $r=z-j\ge0$ and $\rho=\zeta+j>d/2$, $j=0,\dots,z$, and $\|v\|_{\Hzzeta}\le R\Rightarrow \|v\|_{H^{r,\rho}}\le R$,
	imply, for the algebra properties of the weighted Sobolev spaces,
	$$
		\| v^n \|_{H^{r,\rho}}\le C_{nr\rho} \| v^n \|_{H^{nr,\rho}}\le 
		C_{nr\rho} \|v \|_{H^{r,\rho}}^n\le C_{nr\rho} R^{n-1}\|v\|_{H^{r,\rho}},
	$$
	and \eqref{eq:powerliploc} immediately follows, by the definition \eqref{eq:normHpint} of the $\Hzzeta$ norm.
	The second requirement in Definition \ref{def:lip} follows by similar considerations, by the Mean Value Theorem on $g$.
\end{example}

Comparing our result with \cite{DbD99, DbD2002, DbD2005, oh}, we observe that there the focus is on the flat Schrödinger operator $P=i\partial_t-\triangle$, the noise is a real-valued Gaussian process and the solution takes values in some $L^p$-modeled Sobolev spaces, not necessarily of Hilbert type. Conversely, here we deal with the Schrödinger operator $P=i\partial_t-\triangle_\mathfrak g$, associated with an asymptotically flat metric $\mathfrak g$, we allow the noise to be a distribution-valued Gaussian process, and we look for solutions in certain weighted $L^2$-modeled Sobolev spaces. Moreover, the assumptions in  \cite{DbD99, DbD2002, DbD2005, oh} require that the noise is of Hilbert-Schmidt (or radonifying) type, while here we provide a condition on the noise (precisely, on its spectral measure) so that such property holds true, in analogy with the approach we followed in \cite{ACS19b, semilinearpara}, inspired by \cite{peszat}. 
 
The paper is organized as follows. In Section \ref{sec:stochastics} we recall the basic elements of the stochastic 
integration that we need; these materials have appeared, in slightly different forms, e.g. in \cite{ACS19b,  semilinearpara}. Section \ref{sec:nonlin} is devoted to proving our main result.

\section*{Acknowledgement}{This research has been partially supported by the first author's INdAM-GNAMPA Project 2020.}

%
\section{Stochastic integration.}\label{sec:stochastics}

\subsection{Stochastic integration with respect to a cylindrical Wiener process.} 

\begin{definition}\label{cWp}
  Let $Q$ be a self-adjoint, nonnegative definite and bounded linear operator on a separable Hilbert space $H$. An $H$-valued stochastic process $W = \{W_t(h); h\in H, t\geq0\}$ is called a {\em cylindrical Wiener process on $H$} on the complete probability space $(\Omega,\scrF,\P)$ if the following conditions are fulfilled:
  \begin{enumerate}
    \item for any $h\in H$, $\{W_t(h); t\geq0\}$ is a one-dimensional Brownian motion with variance $t\langle Qh,h\rangle_H$;
    \item for all $s,t\geq0$ and $g,h\in H$,
    \[ \E[W_s(g)W_t(h)] = (s\wedge t)\langle Qg,h\rangle_H. \]
  \end{enumerate}
  If $Q=Id_H$, then $W$ is called a standard cylindrical Wiener process.
\end{definition}

Let $\scrF_t$ be the $\sigma$-field generated by the random variables $\{W_t(h); 0\leq s\leq t, h\in H\}$ and the $\P$-null sets. The predictable $\sigma$-field is then the $\sigma$-field in $[0,T]\times\Omega$ generated by the sets $\{(s,t]\times A, A\in\scrF_t, 0\leq s<t\leq T\}$.

We define $H_Q$ to be the completion of the Hilbert space $H$ endowed with the inner product
\[ \langle g,h\rangle_{H_Q} := \langle Qg,h\rangle_H, \]
for $g,h\in H$. In the sequel, we let $\{v_k\}_{k\in\N}$ be a complete orthonormal basis of $H_Q$. Then, the stochastic integral of a predictable, square-integrable stochastic process with values in $H_Q$, $u\in L^2([0,T]\times\Omega; H_Q)$, is defined as
\[ \int_0^t u(s)dW_s := \sum_{k\in\N} \langle u,v_k\rangle_{H_Q} dW_s(v_k). \]
In fact, the series in the right-hand side converges in $L^2(\Omega,\scrF,\P)$ and its sum does not depend on the chosen orthonormal system
$\{v_k\}_{k\in\N}$. Moreover, the It\^o isometry
\[ \E\bigg[\bigg(\int_0^t u(s)dW_s\bigg)^2\bigg] = \E\bigg[\int_0^t \|u(s)\|_{H_Q}^2 ds\bigg] \]
holds true for any $u\in L^2([0,T]\times\Omega;H_Q)$.

This notion of stochastic integral can also be extended to operator-valued integrands. Let $\caH$ be a separable Hilbert space and consider $L_2(H_Q,\caH)$, the space of Hilbert-Schmidt operators from $H_Q$ to 
$\caH$. With this we can define the space of integrable processes (with respect to $W$) as the set of $\scrF$-measureable processes in $L^2([0,T]\times\Omega;L_2(H_Q,\caH))$. Since one can identify the Hilbert-Schmidt operators in $L_2(H_Q,\caH)$ with $\caH\otimes H_Q^*$, one can define the stochastic integral for any $u\in L^2([0,T]\times\Omega;L_2(H_Q,\caH))$ coordinatewise in $\caH$. Moreover, it is possible to establish an It\^o isometry, namely,
\beqs\label{isomhilb}
\E\Bigg[\bigg\|\int_0^t u(s)dW_s\bigg\|_\caH^2\Bigg] := \int_0^t \E\big[\|u(s)\|_{L_2(H_Q,\caH)}^2\big] ds.
\eeqs

\subsection{The noise term}\label{sec:noise}
 In this paper we consider a distribution-valued
Gaussian process $\{\Xi(\phi);\; \phi\in\mathcal{C}_0^\infty(\mathbb{R}_+\times\Rd)\}$ on a complete probability space $(\Omega, \scrF, \P)$,
with mean zero and covariance functional given by
\begin{equation}
	\E[\Xi(\phi)\Xi(\psi)] = \int_0^\infty\int_\Rd \big(\phi(t)\ast\tilde{\psi}(t)\big)(x)\,\Gamma(dx) dt,
	\label{eq:correlation}
\end{equation}
where $\tilde{\psi}(t,x) := (2\pi)^{-d}\,\overline{\psi(t,-x)}$, $\ast$ is the convolution operator and $\Gamma$ is a nonnegative, nonnegative definite, tempered measure on $\Rd$ usually called \emph{correlation measure}.
Then \cite[Chapter VII, Th\'{e}or\`{e}me XVIII]{schwartz} implies that there exists a nonnegative tempered measure $\mathfrak{M}$ on $\Rd$, usually called \emph{spectral measure}, such that $\caF\Gamma = \widehat{\Gamma}=\mathfrak{M}$, where $\caF$ and $\widehat{\phantom d}$ denote the Fourier transform.
By Parseval's identity, the right-hand side of \eqref{eq:correlation} can be rewritten as
\begin{equation*}
	\E[\Xi(\phi)\Xi(\psi)] = \int_0^{\infty}\int_{\Rd}[\caF\phi(t)](\xi)\
	\cdot
	\overline{[\caF\psi(t)](\xi)}\,\mathfrak{M}(d\xi) dt.
\end{equation*}
\subsection{The Cameron-Martin space associated with the noise}\label{sec:CMspace}
To give a meaning to the stochastic integral appearing in \eqref{eq:mildsolutionSPDE}
as a stochastic integral with respect to a cylindrical Wiener process on a suitable Hilbert space, we need to understand the noise $\Xi$ in terms of a canonically associated Hilbert space $\mathcal H_\Xi$, the so-called Cameron-Martin space associated with $\Xi$.
\begin{definition}
The Cameron-Martin space associated with $\Xi$ is the set
\begin{equation}\label{cammart}
\mathcal H_\Xi=\{\widehat{\varphi\mathfrak{M}}\colon\varphi\in L^2_{\mathfrak{M},s}(\R^d)\},
\end{equation}
where $L^2_{\mathfrak{M},s}(\R^d)$ is the space of symmetric functions in $L^2_\mathfrak{M}(\R^d)$, i.e.
$\check{\varphi}(x)=\varphi(-x)=\varphi(x)$, $x\in\R^d$, and $\ds\int_{\R^d}|\varphi(x)|^2\,\mathfrak{M}(dx)<\infty$.
\end{definition}
Clearly, $\mathcal H_\Xi\subset\mathcal S'(\R^d)$. The Cameron-Martin space $\mathcal H_\Xi$, endowed with the inner product
\[
\langle\widehat{\varphi\mathfrak{M}},\widehat{\psi\mathfrak{M}}\rangle_{\mathcal H_\Xi}:=\langle\varphi,\psi\rangle_{L^2_{\mathfrak{M},s}(\R^d)}, \quad \forall\varphi,\psi\in L^2_{\mathfrak{M},s}(\R^d),
\]
and the corresponding norm \[||\widehat{\varphi\mathfrak{M}}||_{\mathcal H_\Xi}^2=||\varphi||_{L^2_{\mathfrak{M},s}(\R^d)}^2,\]
turns out to be a real separable Hilbert space, see \cite[Propostition 2.1]{peszat}. Thus, $\Xi$ is a cylindrical  Wiener process $W$
on $(\mathcal H_\Xi,\langle\cdot,\cdot\rangle_{\mathcal H_\Xi})$ which takes values in any Hilbert space $\caH$ such that the embedding $\caH_\Xi\hookrightarrow \caH$ is an Hilbert-Schmidt map.
%
%

\subsection{The concept of solution to \eqref{eq:SPDE}}\label{sec:rigorous_solution}
With the preparation above, we can now give the precise meaning of solution to \eqref{eq:SPDE}.
\begin{definition}\label{def:rigorous_solution}
We call \textit{(mild) function-valued solution to \eqref{eq:SPDE}} an $L^2(\Omega,\caH)$-family of random elements $u(t)$, 
satisfying the stochastic integral equation
\begin{equation}
\label{eq:rigorous_solution}
	u(t) = S(t)u_0 -i \int_0^t S(t-s)\gamma(s,u(s))ds - i\int_0^t S(t-s)\sigma(s,u(s))dW_s,
\end{equation}
for all $t\in[0,T_0]$, $x\in\R^d$, where $T_0\in(0, T]$ is a suitable time horizon, 
$v_0(t):=S(t)u_0\in\caH$ for $t\in[0,T_0]$, $S(t)$ is the propagator of $P$ provided by Theorem \ref{T14craig}, $\gamma$ and $\sigma$ are nonlinear operators defined by the so-called Nemytskii operators associated 
with the functions $\gamma$ and $\sigma$ in \eqref{eq:SPDE}, and $W$ is the cylindrical Wiener process associated with $\Xi$, defined in Section \ref{sec:CMspace}.
\end{definition}

\begin{remark}\label{rem:peslem} 
It follows that, to ensure that the stochastic integral appearing in \eqref{eq:rigorous_solution} makes sense, it will be enough to verify that 
$\{S(t-\cdot)\sigma(\cdot,u(\cdot))\in L^2([0,t]\times\Omega; L_2(\mathcal H_\Xi, \caH)$ for a suitable separable Hilbert space  $\caH$.
\end{remark}




%
\section{function-valued solutions for semilinear SPDEs of Schr\"odinger type.}\label{sec:nonlin}

This section is devoted to precisely state and prove the main result of this paper, which provides the existence and uniqueness of a function-valued solution \eqref{eq:rigorous_solution} to the Cauchy problem \eqref{eq:SPDE} for a Schr\"odinger type equation on a curved space, under suitable conditions on the coefficients of $P$, the drift and diffusion term, the initial data, the noise $\Xi$ (i.e., its spectral measure $\mathfrak{M}$). The statement is the next Theorem \ref{main}.
\begin{theorem}\label{main}
Let us consider the Cauchy problem \eqref{eq:SPDE} for a Schr\"odinger type operator under assumptions (1)-(5)
with $u_0\in \HzzetaRd$, $z,\zeta\in\N$.
	Assume that $\gamma,\sigma\in\mathrm{Lip_{loc}}(z,\zeta)$ in some open subset  $U\subset \HzzetaRd$ with $u_0 \in U$, and
	\begin{equation}\label{eq:meascm2}
  		\int_\Rd\mathfrak{M}(d\xi)< \infty.
	\end{equation}
	%
%
Then, there exists a time horizon $0< T_0\leq T$ such that the Cauchy problem \eqref{eq:SPDE} admits a unique
 solution $u\in L^2([0,T_0]\times\Omega, \HzzetaRd)$ in the sense of Definition \ref{def:rigorous_solution}. That is, 
$u$ satisfies \eqref{eq:rigorous_solution} for all $t\in[0,T_0]$, where $S(t)$ is the propagator of $P$, the first integral in \eqref{eq:rigorous_solution} is a Bochner integral, and the second integral in \eqref{eq:rigorous_solution} is understood as stochastic integral of the $\HzzetaRd$-valued stochastic process $S(t-\ast)\sigma(\ast,u(\ast))$ with respect to the stochastic noise $\Xi$.
\end{theorem}

The key result to prove Theorem \ref{main} is the next Lemma \ref{lem:weightedpesz2}, a variant of a result due, in its original form given for wave-type equations, to Peszat \cite{peszat}.

\begin{lemma}\label{lem:weightedpesz2}
Let $\sigma\in \Lip(z,\zeta)$, $z,\zeta\in\N$, and let $S(t)$ be the propagator of $P$ provided by Theorem \ref{T14craig}. 
If  the spectral measure satisfies
\eqref{eq:meascm2}, then, for every $w\in \HzzetaRd$, the operator 
\[
\Phi(t,s)=\colon\psi\mapsto S(t-s)\sigma(s,\cdot,w) \psi
\] 
belongs to $L_2(\mathcal H_\Xi, \HzzetaRd)$.
Moreover, the Hilbert-Schmidt norm of $\Phi(t,s)$ can be estimated by
\begin{equation}\label{battezzata2}
\|\Phi(t,s)\|_{L_2(\mathcal H_\Xi, \HzzetaRd)}^2\lesssim
C_{t,s}\left[1+\|w\|_{\HzzetaRd}\right]^2
\int_\Rd \mathfrak{M}(d\xi),
\end{equation}
for $C_{t,s}=e^{2C_{z,\zeta}t}C_s^2$, where $C_s$ is the constant in Definition \ref{def:lip} and $C_{z,\zeta}$ is the constant in Theorem \ref{T14craig}.
\end{lemma}

\begin{remark} Lemma \ref{lem:weightedpesz2} shows that the multiplication operator $\mathcal H_\Xi\ni\psi\mapsto S(t-s)\sigma(s,u)\cdot \psi$ is  Hilbert-Schmidt from $\mathcal H_\Xi$ to $\caH=\Hzzeta$ under suitable assumptions on $\sigma$. Therefore, the noise $\Xi$ defines a cylindrical Wiener process on
$(\mathcal H_\Xi,\langle\cdot,\cdot\rangle_{\mathcal H_\Xi})$ with values in $\Hzzeta$, and the third summand in the right-hand side of \eqref{eq:rigorous_solution} is a well-defined stochastic integral with respect to a cylindrical  Wiener process on 
$(\mathcal H_\Xi,\langle\cdot,\cdot\rangle_{\mathcal H_\Xi})$ which takes values in
$\Hzzeta$ (see Remark \ref{rem:peslem}). 
To prove Theorem \ref{main}, we will rely on the mapping properties of the fundamental solution $S(t)$ on the $\Hzzeta$ spaces, illustrated in Remark \ref{S}, and on a fixed point scheme.
\end{remark}

\begin{proof}[Proof of Lemma \ref{lem:weightedpesz2}]

We fix an orthonormal basis $\{e_j\}_{j\in\N}=\{\widehat{f_j\mathfrak{M}}\}_{j\in\N}$ of $\mathcal H_\Xi$, where $\{f_j\}_{j\in\N}$ is an orthonormal basis in $L^2_{\mathfrak{M},s}$. Using the definition of Hilbert-Schmidt norm and the continuity of the map $S(t): \Hzzeta\longrightarrow \Hzzeta$, we compute
\begin{equation}\label{las}
\begin{aligned}
||\Phi(t,s)||_{L_2(\mathcal H_\Xi, \Hzzeta)}^{2}&=\sum_{j\in\N}||S(t-s)\sigma(s, \cdot, w)\widehat{f_j\mathfrak{M}}||_{\Hzzeta}^2
\leq 
e^{2C_{z,\zeta}t}\sum_{j\in\N}||\sigma(s, \cdot, w)\widehat{f_j\mathfrak{M}}||_{\Hzzeta}^2
\\
&\lesssim e^{2C_{z,\zeta}t}\sum_{j\in\N}\sum_{k=0}^r||\langle \cdot\rangle^{z-k}\langle D\rangle^{k+\zeta}\sigma(s, \cdot, w)\widehat{f_j\mathfrak{M}}||_{L^2}^2
\\
&=  e^{2C_{z,\zeta}t}(2\pi)^{-d}\sum_{k=0}^r\sum_{j\in\N}\int_{\R^d}\left\vert\mathcal F\left(\langle \cdot\rangle^{z-k}\langle D\rangle^{k+\zeta}\sigma(s, \cdot, w)\widehat{f_j\mathfrak{M}}\right)\right\vert^2(\xi)d\xi,
\end{aligned}
\end{equation}
where $C_{z,\zeta}$ is the costant appearing in Theorem \ref{T14craig}. Now, using the well-known fact that the Fourier transform of a product is the ($(2\pi)^{-d}$ multiple of the)
convolution of the Fourier transforms, the
property $f_j(-x)=f_j(x)$ (by the definition of $L^2_{\mathfrak{M},s}$), that
$\{f_j\}$ is an orthonormal system in $L^2_{\mathfrak{M}}$, and Bessel's inequality,
we get
\beqsn
\sum_{j\in\N}&&\hskip-0.3cm\left\vert\mathcal F\left(\langle \cdot\rangle^{z-k}\langle D\rangle^{k+\zeta}\sigma(s, \cdot, w)\widehat{f_j\mathfrak{M}}\right)\right\vert^2(\xi)
\\
&=&(2\pi)^{-d}\sum_{j\in\N}|\mathcal F\left(\langle \cdot\rangle^{z-k}\langle D\rangle^{k+\zeta}\sigma(s, \cdot, w)\right)\ast\widehat{\widehat{f_j\mathfrak{M}}}|^2(\xi)
\\
&=&\sum_{j\in\N}|\mathcal F\left(\langle \cdot\rangle^{z-k}\langle D\rangle^{k+\zeta}\sigma(s, \cdot, w)\right)\ast f_j\mathfrak{M}|^2(\xi)
\\
&=&\sum_{j\in\N}\left\vert\int_{\R^d}\left[\mathcal F\left(\langle \cdot\rangle^{z-k}\langle D\rangle^{k+\zeta}\sigma(s, \cdot, w)\right)\right](\xi-\eta) f_j(\eta)\mathfrak{M}(d\eta)\right\vert^2
\\
&\le &\int_{\R^d}\left\vert\mathcal F\left(\langle \cdot\rangle^{z-k}\langle D\rangle^{k+\zeta}\sigma(s, \cdot, w)\right)\right\vert^2(\xi-\eta)\mathfrak{M}(d\eta).
\\
\eeqsn
Inserting this in \eqref{las}, and using the continuity of $\langle \cdot\rangle^{z-k}\langle D\rangle^{k+\zeta}$ on the Sobolev-Kato spaces, we finally get:
\begin{equation*}
\begin{aligned}
||\Phi(t,s)||_{L_2(\mathcal H_\Xi, \Hzzeta)}^{2}
&\lesssim e^{2C_{z,\zeta}t}(2\pi)^{-d}\sum_{k=0}^z
\int_{\R^d}\int_{\R^d}\left\vert\mathcal F\left(\langle \cdot\rangle^{z-k}\langle D\rangle^{k+\zeta}\sigma(s, \cdot, w)\right)\right\vert^2(\xi-\eta)\mathfrak{M}(d\eta)d\xi
\\
&= e^{2C_{z,\zeta}t}(2\pi)^{-d}\sum_{k=0}^z\int_{\R^d}\int_{\R^d}\left\vert\mathcal F\left(\langle \cdot\rangle^{z-k}\langle D\rangle^{k+\zeta}\sigma(s, \cdot, w)\right)\right\vert^2(\theta)\mathfrak{M}(d\eta)d\theta
\\
&= e^{2C_{z,\zeta}t}(2\pi)^{-d}\sum_{k=0}^z\left(\int_{\R^d}\mathfrak{M}(d\eta)\right)\int_{\R^d}\left\vert\mathcal F\left(\langle \cdot\rangle^{z-k}\langle D\rangle^{k+\zeta}\sigma(s, \cdot, w)\right)\right\vert^2(\theta)d\theta
\\
&= e^{2C_{z,\zeta}t}(2\pi)^{-d}\int_{\R^d}\mathfrak{M}(d\eta)\sum_{k=0}^z\|\mathcal F(\langle \cdot\rangle^{z-k}\langle D\rangle^{k+\zeta}\sigma(s, \cdot, w))\|_{L^2}^2
\\
&\lesssim e^{2C_{z,\zeta}t}\int_{\R^d}\mathfrak{M}(d\eta)\sum_{k=0}^z\|\langle \cdot\rangle^{z-k}\langle D\rangle^{k+\zeta}\sigma(s, \cdot, w)\|_{L^2}^2
\\
&= e^{2C_{z,\zeta}t}\int_{\R^d}\mathfrak{M}(d\eta)\sum_{k=0}^z\|\sigma(s, \cdot, w)\|_{H^{z-k,k+\zeta}}^2
\\
&\lesssim e^{2C_{z,\zeta}t}\|\sigma(s, \cdot, w)\|_{\mathcal H_{z,\zeta}}^2\int_{\R^d}\mathfrak{M}(d\eta)
\\
&\leq e^{2C_{z,\zeta}t}C_s^2\left(1+\|w\|_{\mathcal H_{z,\zeta}}\right)^2\int_{\R^d}\mathfrak{M}(d\eta),
\end{aligned}
\end{equation*}
where $C_s$ is the constant in Definition \ref{def:lip} and is a continuous function with respect to $s\in [0,T]$.
\end{proof}

\noindent
We are now ready to prove our main result.

\begin{proof}[Proof of Theorem \ref{main}]

Inserting the right-hand side $f(s,u(s))=\gamma(s,u(s)) + \sigma(s,u(s))\dot{\Xi}(s)$ in the solution 
$$
	u(t) = S(t)u_0-i\int_0^tS(t-s)f(s,u(s))ds
$$
of the associated deterministic Cauchy problem with data $f, u_0$, and operator $P$, we can formally construct the ``mild solution" $u$ to \eqref{eq:SPDE}, namely
\[
	\begin{aligned}
	u(t)&=S(t)u_0-i\int_0^tS(t-s)\gamma(s,u(s))ds
	-i\int_0^tS(t-s)\sigma(s,u(s))\dot{\Xi}(s)ds.
	\end{aligned}
\]
By the linear deterministic theory in \cite{craig}, we know that
\begin{equation}\label{eq:regv0}
	v_0(t):=S(t)u_0\in C([0,T], \Hzzeta),
\end{equation}
since, by hypotheses, $u_0\in \Hzzeta$. We then consider the map $u\to \mathcal{T}u$ on 
$L^2([0,T_0]\times\Omega, \Hzzeta)$, $T_0\in(0,T]$ small enough, defined as
\beqs\label{calm}
\hskip-0.8cm \mathcal{T}u(t)&:=&
v_0(t)-i\ds\int_0^t S(t-s)\gamma(s,u(s))ds-i\ds\int_0^t S(t-s)\sigma(s,u(s))dW_s
\\
\nonumber
&:=&v_0(t)+\mathcal{T}_1u(t)+\mathcal{T}_2u(t), \quad t\in[0,T_0], 
\eeqs
where the last integral on the right-hand side is understood as the stochastic integral of the stochastic process $S(t-*)\sigma(*,\cdot,u(*,\cdot))\in L^2([0,T_0]\times\Omega, \Hzzeta)$ with respect to the cylindrical Wiener process $\{W_t(h)\}_{t\in[0,T], h\in \Hzzeta}$,
associa\-ted with the random noise $\Xi(t)$, which is well-defined and takes values in $\Hzzeta$, by Lemma \ref{lem:weightedpesz2}.

We initially work under the stronger assumption $\gamma,\sigma\in\Lip(z,\zeta)$. To prove that the mild solution \eqref{eq:rigorous_solution} 
of the Cauchy problem \eqref{eq:SPDE} indeed exists and has the stated properties, it is enough to check that
\[
	\mathcal{T}\colon L^2([0,T_0]\times\Omega, \Hzzeta)\longrightarrow L^2([0,T_0]\times\Omega,\Hzzeta)
\]
is well-defined, it is Lipschitz continuous on $L^2([0,T_0]\times\Omega, \Hzzeta)$, and it becomes a strict contraction if we take $T_0\in(0,T]$ small enough. Then, an application of Banach's fixed point Theorem will provide the existence of a unique solution $u\in L^2([0,T_0]\times\Omega, \Hzzeta)$ satisfying $u=\mathcal{T}u$, that is \eqref{eq:rigorous_solution}.

\vskip+0.2cm

We first check that $\mathcal{T}u\in L^2([0,T_0]\times\Omega, \Hzzeta)$ for every $u\in L^2([0,T_0]\times\Omega, \Hzzeta)$. We have:
\begin{itemize}
\item[-] $v_0\in C([0,T_0], \Hzzeta)\hookrightarrow L^2([0,T_0]\times\Omega, \Hzzeta)$;
\item[-] $\mathcal{T}_1u$ is in $L^2([0,T_0]\times\Omega, \Hzzeta)$; indeed, $\mathcal{T}_1u(t)$ is defined as the Bochner integral on $[0,t]$ of the function
$s\mapsto -iS(t-s)\gamma(s,\cdot,u(s))$, which takes values in $L^2(\Omega, \Hzzeta)$; by the properties of Bochner integrals, the continuity of 
$S(t-s)$ on the $\Hzzeta$ spaces, and the assumption $\gamma\in \Lip(z,\zeta)$, we see that 
\beqs\nonumber
\|\mathcal{T}_1u\|_{L^2([0,T_0]\times\Omega, \Hzzeta)}^2&&=\ds\E\left[\int_0^{T_0}\|\mathcal{T}_1u(t)\|_{\Hzzeta}^2\,dt\right]
=
\ds\int_0^{T_0}\E\left[\,\left\|\ds\int_0^t S(t-s)[\gamma(s,\cdot,u(s))]ds\right\|_{\Hzzeta}^2\,\right]dt
\\\nonumber
&&\leq\ds\int_0^{T_0}\!\!\ds\int_0^t \E\left[\,\left\|S(t-s)[\gamma(s,\cdot,u(s))]\right\|_{\Hzzeta}^2\,\right]dsdt
\lesssim \ds\int_0^{T_0}\!\!\ds\int_0^t e^{2C_{z,\zeta}(t-s)} \,\E\left[\,\left\|\gamma(s,\cdot,u(s))\right\|_{\Hzzeta}^2\,\right] dsdt
\\\nonumber
&&
\leq \ds\int_0^{T_0}\!\!\ds\int_0^t e^{2C_{z,\zeta}(t-s)}C_s^2\,\E\left[\left(1+\|u(s)\right\|_{\Hzzeta})^2\right]dsdt
\leq T_0\left(\max_{0\leq s\leq t\leq T_0}e^{2C_{z,\zeta}(t-s)}C_s^2\right)
\!\ds\int_0^{T_0}\E\left[\left(1+\|u(s)\right\|_{\Hzzeta})^2\right]ds
\\\label{uguale}
&&\lesssim T_0C_{T_0}\left[T_0+\|u\|^2_{L^2([0,T_0]\times\Omega, \Hzzeta)}\right]<\infty,
\eeqs
where $C_{T_0}$ depends continuously on $T_0$;
\\
\item[-] $\mathcal{T}_2u$ is in $L^2([0,T_0]\times\Omega, \Hzzeta)$, in view of the fundamental isometry \eqref{isomhilb},
Lemma \ref{lem:weightedpesz2}, and the fact that the expectation can be moved inside and outside time integrals, by Fubini's Theorem:
\beqs
\nonumber
\|\mathcal{T}_2u\|_{L^2([0,T_0]\times\Omega, \mathcal H_{z,\zeta})}^{2}&=&\E\left[\ds\int_0^{T_0}\|\mathcal{T}_2u(t)\|_{\Hzzeta}^2dt\right]
=\ds\int_0^{T_0}\E\left[\left\|\ds\int_0^t S(t-s)\sigma(s,\cdot,u(s))dW_s\right\|_{\Hzzeta}^2\right]dt
\\
\nonumber
&=&\ds\int_0^{T_0}\ds\int_0^t \E\left[\left\|S(t-s)\sigma(s,\cdot,u(s))\right\|_{L_2(\mathcal H_\Xi, \Hzzeta)}^2\right]dsdt
\\
\nonumber
&\lesssim&\ds\int_0^{T_0}\ds\int_0^t \E\left[ C_{t,s}\left(1+\|u(s)\|_{\Hzzeta}\right)^2 \int_\Rd \mathfrak{M}(d\xi)\right]dsdt
\\
\nonumber
&=&\int_\Rd \mathfrak{M}(d\xi)\cdot \left(\max_{0\leq s\leq t\leq T_0} C_{t,s}\right)\cdot T_0\cdot 
\ds\int_0^{T_0} \E\left[  \left(1+\|u(s)\|_{\Hzzeta}\right)^2
\right]ds
\\
\nonumber
&\lesssim &T_0C_{T_0} \int_\Rd \mathfrak{M}(d\xi) 
\left(T_0+\ds\int_0^{T_0} \E\left[\left\|u(s)\right\|_{\Hzzeta}^2\right]ds\right)
\\
\label{eq:estT2}
&=&T_0C_{T_0} \int_\Rd \mathfrak{M}(d\xi) \cdot\left[T_0+\|u\|^2_{L^2([0,T_0]\times\Omega, \Hzzeta)}\right]<\infty,
\eeqs
with $C_{T_0}$ continuous with respect to $T_0$.
\end{itemize}

\vskip+0.2cm
Now, we show that $\mathcal T$ is a contraction for $T_0\in(0,T]$ suitably small. We take $u_1,u_2\in L^2([0,T_0]\times\Omega, \Hzzeta)$ and compute
\beqs\nonumber
\|\mathcal{T}u_1&-&\mathcal{T}u_2\|_{L^2([0,T_0]\times\Omega, \Hzzeta)}^2
\\\nonumber
&\leq& 2\left(\|\mathcal{T}_1u_1-\mathcal{T}_1u_2\|_{L^2([0,T_0]\times\Omega, \mathcal H_{z,\zeta})}^2+\|\mathcal{T}_2u_1-
\mathcal{T}_2u_2\|_{L^2([0,T_0]\times\Omega, \Hzzeta)}^2\right)
\\\label{pri2}
&=&2\ds\int_0^{T_0}\E\left[\,\left\|\ds\int_0^t S(t-s)(\gamma(s,\cdot,u_1(s))-\gamma(s,\cdot,u_2(s)))ds\right\|_{\Hzzeta}^2\,\right]dt
\\\label{se2}
&+&2\ds\int_0^{T_0}\E\left[\,\left\|\ds\int_0^t S(t-s)(\sigma(s,\cdot,u_1(s))-\sigma(s,\cdot,u_2(s)))dW_s\right\|_{\Hzzeta}^2\,\right]dt.
\eeqs
To estimate the term \eqref{pri2}, we first move the expectation and the $\Hzzeta$-norm inside the integral with respect to $s$.
Then, by continuity of $S(t-s)$ on the $\Hzzeta$ spaces and the second requirement in Definition \ref{def:lip}, we obtain
\beqsn%
\ds\int_0^{T_0}\E&&\left[\,\left\|\ds\int_0^t S(t-s)(\gamma(s,\cdot,u_1(s))-\gamma(s,\cdot,u_2(s)))ds\right\|_{\Hzzeta}^2\,\right]dt
\\
&&\lesssim\ds\int_0^{T_0}\ds\int_0^t \E\left[\,\left\|S(t-s)(\gamma(s,\cdot,u_1(s))-\gamma(s,\cdot,u_2(s)))\right\|_{\Hzzeta}^2\,\right]dsdt
\\
&&\leq \ds\int_0^{T_0}\ds\int_0^t e^{2C_{z,\zeta}(t-s)}\,\E\left[\,\left\|\gamma(s,\cdot,u_1(s))-\gamma(s,\cdot,u_2(s))\right\|_{\mathcal H_{z,\zeta}}^2\,\right]dsdt
\\
&&\leq \ds\int_0^{T_0}\ds\int_0^t e^{2C_{z,\zeta}(t-s)}C_s^2\,\E\left[\,\left\|u_1(s)-u_2(s)\right\|_{\mathcal H_{z,\zeta}}^2\,\right]dsdt
\\
&&\leq C_{T_0}T_0\ds\int_0^{T_0} \E\left[\,\left\|u_1(s)-u_2(s)\right\|_{\mathcal H_{z,\zeta}}^2\,\right]ds
\\
&&=C_{T_0}T_0\|u_1-u_2\|^2_{L^2([0,T_0]\times\Omega, \mathcal H_{z,\zeta})},
\eeqsn
with $C_{T_0}$ continuous with respect to $T_0$. To estimate the term \eqref{se2} we apply, 
here below, the fundamental isometry \eqref{isomhilb} to pass from the first to the second line, 
the second requirement in Definition \ref{def:lip} and the analog of formula \eqref{battezzata2} from Lemma \ref{lem:weightedpesz2},
with $\sigma(s,\cdot,u_1(s))-\sigma(s,\cdot,u_2(s))$ in place of $\sigma(s,\cdot,w)$, to pass from the second to the third line, and finally obtain
\beqsn
\ds\int_0^{T_0}&&\hskip-0.3cm\E\left[\left\|\ds\int_0^t S(t-s)(\sigma(s,\cdot,u_1(s))-\sigma(s,\cdot,u_2(s)))dW_s\right\|_{\mathcal H_{z,\zeta}}^2\right]dt
\\
&&=\ds\int_0^{T_0}\ds\int_0^t \E\left[\left\|S(t-s)(\sigma(s,\cdot,u_1(s))-\sigma(s,\cdot,u_2(s)))\right\|_{L_2(\mathcal H_\Xi, \mathcal H_{z,\zeta})}^2\right]dsdt
\\
\\
&&\lesssim \ds\int_0^{T_0}\ds\int_0^t \E\left[
C_{t,s}\|u_1(s)-u_2(s)\|_{\mathcal H_{z,\zeta}}^2\int_\Rd {\mathfrak{M}(d\xi)}\right]dsdt
\\
&&\leq \int_\Rd {\mathfrak{M}(d\xi)}
\ds\int_0^{T_0}\ds\int_0^t C_{t,s}\E\left[\left\|u_1(s)-u_2(s)\right\|_{\mathcal H_{z,\zeta}}^2\right]dsdt
\\
&&\lesssim\int_\Rd {\mathfrak{M}(d\xi)}
\cdot C_{T_0}T_0\cdot \|u_1-u_2\|^2_{L^2([0,T_0]\times\Omega, \mathcal H_{z,\zeta})},
\eeqsn
with $C_{T_0}$ continuous with respect to $T_0$. Summing up, we have proved that
\beqsn
\|\mathcal{T}u_1-\mathcal{T}u_2\|_{L^2([0,T_0]\times\Omega, \mathcal H_{z,\zeta})}^2\leq
C_{T_0}T_0\left(1+\int_\Rd {\mathfrak{M}(d\xi)}\right)
\cdot \|u_1-u_2\|^2_{L^2([0,T_0]\times\Omega, \mathcal H_{z,\zeta})} \,,
\eeqsn
that is, $\mathcal{T}$ is Lipschitz continuous on $L^2([0,T_0]\times\Omega, \mathcal H_{z,\zeta})$. Moreover, in view of the assumption \eqref{eq:meascm2} and the
continuity of $C_{T_0}$ with respect to $T_0$, we can take $T_0>0$ so small that
\[
	C_{T_0}T_0\left(1+\int_\Rd {\mathfrak{M}(d\xi)}\right)< 1,
\]
making $\mathcal{T}$ a strict contraction on $L^2([0,T_0]\times\Omega, \mathcal H_{z,\zeta})$, so that it admits a unique fixed point $u=\mathcal{T}u$,
$u\in L^2([0,T_0]\times\Omega, \mathcal H_{z,\zeta})$, as claimed.

\vskip+0.2cm
Finally, when $\gamma,\sigma\in\Liploc(z,\zeta)$, we first observe that there exists 
$\BBB$, closed ball of radius $R>0$, centred in $u_0$, such that $\BBB\subset U$. Since, of course, 
\[
	v_0\in C([0,T_0], \Hzzeta) \text{ and } S(0)u_0=u_0 \Rightarrow \exists T_0\in(0,T] \colon  \|S(t)u_0-u_0\|_{\Hzzeta}\le \frac{R}{2}, \; t\in[0,T_0],
\]
by computations similar to those employed to show \eqref{uguale} and \eqref{eq:estT2}, it turns out that, choosing $T_0\in(0,T]$ suitably small,
\[
	\|\mathcal{T}_1u(t)\|_{\Hzzeta} + \|\mathcal{T}_2u(t)\|_{\Hzzeta}
	\le 
	\left(T_0 C_{T_0}\right)^\frac{1}{2}(1+\|u_0\|_{\Hzzeta}+R)\left[\left(\int_{\R^d}\mathfrak{M}(d\xi)\right)^\frac{1}{2}+1\right]\le\frac{R}{2}, 
	\; u(t)\in\BBB, t\in[0,T_0], 
\]
so that $\mathcal{T}\colon L^2([0,T_0]\times\Omega, \BBB)\to L^2([0,T_0]\times\Omega, \BBB)$ and it is a strict contraction there.

\vskip+0.2cm

The proof is complete.
\end{proof}

\begin{remark}We conclude remarking that it is possible, in the linear case (that is, when drift and diffusion do not depend on $u$), to construct also a random-field solution $u_{\mathrm{rf}}$ of the Cauchy problem \eqref{eq:SPDE}, under condition \eqref{eq:meascm2}.
Explicitly, $u_{\mathrm{rf}}$ is defined as a map associating a random variable with each
$(t,x)\in[0,T_0]\times\Rd$, where $T_0>0$ is the time horizon of the solution. This is achieved by the properties of the fundamental solution to the operator $P$ from Theorem \ref{T14craig}, following the scheme developed in \cite{ACS19a,linearpara,alessiandre}. Indeed, random-field solutions can be defined by means of a martingale measure, derived from the random noise $\dot \Xi$: this is the approach due to Walsh and Dalang, see \cite{conusdalang,dalang,walsh}, to interpret stochastic integrals. 
As it is well known, in several cases the two approaches to stochastic integration produce ``the same solutions'', see the comparison paper \cite{dalangquer}.
The random-field solution $u_{\mathrm{rf}}$ turns out to coincide with the function-valued solution $u$ obtained in Theorem \ref{main}, see \cite{ACS19b} for an explicit comparison in the case of hyperbolic SPDEs. The details of the analysis of $u_{\rm{rf}}$ sketched above will appear elsewhere.
\end{remark}

\end{document}